\documentclass[10pt, 
]{amsart}
\usepackage{
amssymb
}
\newcommand{\no}[1]{#1}

\renewcommand{\no}[1]{}

\no{\usepackage{times}\usepackage[
subscriptcorrection, slantedGreek, 
nofontinfo]{mtpro}
\renewcommand{\Delta}{\upDelta}
}

\newtheorem{theorem}{Theorem}[section]

\newtheorem{lemma}{Lemma}[section]

\theoremstyle{remark}
\newtheorem{remark}{Remark}[section]

\newcommand{\eps}{\epsilon}

\newcommand{\R}{\mathbb{R}}

\renewcommand{\r}[1]{(\ref{#1})}
\newcommand{\be}[1]{\begin{equation}\label{#1}}
\newcommand{\ee}{\end{equation}}

\renewcommand{\d}{{d}}

\renewcommand{\i}{{i}}

\newcommand{\bo}{\Gamma}

\numberwithin{equation}{section}

\title[Borg-Levinson theorem]{Stability for the multi-dimensional Borg-Levinson  theorem\\  with partial spectral data}

\author[Mourad Choulli]{Mourad Choulli}
\address{LMAM, UMR 7122, Universit\'e Paul Verlaine-Metz et CNRS, Ile du Saulcy, 57045 Metz cedex 1, France}
\email{choulli@univ-metz.fr}

\author[Plamen Stefanov]{Plamen Stefanov}
\address{Department of Mathematics, Purdue University, West Lafayette, IN 47907, USA}
\thanks{Second author partly supported by a NSF  Grant DMS-0800428}
\email{stefanov@math.purdue.edu}

\begin{document}
\maketitle

\begin{abstract}
We prove a stability estimate related to the multi-dimensional Borg-Levinson theorem of determining a potential from spectral data: the Dirichlet eigenvalues $\lambda_k$ and  the normal derivatives $\partial \phi_k/\partial \nu$ of the eigenfunctions on the boundary of a bounded domain. The estimate is of H\"older type, and we allow finitely many eigenvalues and normal derivatives to be unknown. We also show that if the spectral data is known asymptotically only, up to $O(k^{-\alpha})$ with $\alpha\gg1$, then we still have  H\"older stability. 
 
%
%
\end{abstract}

\section{Introduction} 
In 1988, Nachman, Sylvester, Uhlmann \cite{NachmanSU_Borg-Lev} 
proved an $n$-dimensional version of the classical one dimensional Borg \cite{MR0015185} and Levinson \cite{MR0032067} theorem: one can determine uniquely a potential $q$ in the Schr\"odinger operator $-\Delta+q$, from knowledge of the Dirichlet eigenvalues and the traces of the normal derivatives of the normalized eigenfunctions on the boundary of a bounded domain. The proof is based on relating the spectral data to the Dirichlet-to-Neumann map $\Lambda_{q-\lambda}$ for all frequencies $\lambda$, see \r{8} below. Then we recover the potential by applying $\Lambda_{q-\lambda}$ to high-frequency solutions. In fact, one frequency is enough for uniqueness \cite{SylvesterU87} but the recovery then is only logarithmically stable \cite{Alessandrini88, Mandache}. The link between the spectral data and $\Lambda_{q-\lambda}$ was also noticed by Novikov in   \cite{No}. 

The first stability estimate for this problem of conditional H\"older type was proved by Alessandrini and Sylvester \cite{MR1070844}. A variant of this result was given  by the first author in \cite{Ch} and extended later by Bellassoued and Dos Santos Ferreira \cite{BD} to the case of the Schr\"odinger operator on a simple Riemannian compact manifold with boundary. The main idea in the approach initiated by Alessandrini and Sylvester consists in transforming the stability estimate for the spectral problem into a stability estimate for the inverse problem of determining the potential in a wave equation from the corresponding hyperbolic Dirichlet-to-Neumann map $H_q$. They gave an explicit formula relating $\Lambda _{q-\lambda}$ to $H_q$. On the other hand, reconstruction of a potential from $H_q$ can be done in a stable way using geometric optics \cite{RakeshSymes, Sun_90}. A powerful method based on the boundary control method was initiated by Belishev and developed by Belishev, Katchalov, Kurylev, Lassas and others. We refer to \cite{KKL} for more details. This method proves uniqueness of recovery of the coefficients of a general  elliptic operator from the corresponding hyperbolic Dirichlet-to-Neumann map $H$.  
An important ingredient of this method is the unique continuation result of Tataru \cite{tataru95,Tataru99}  which makes it unlikely to provide H\"older stability when the latter might hold. In the metric case, conditions on the metric are needed apparently, see e.g., \cite{SU-IMRN, BD}. 

In \cite{BCY1}, a logarithmic type stability estimate with a partial hyperbolic Dirichlet-to-Neumann map $\tilde{H}_q$ was proved, under the assumption that the potential is known near the boundary. The proof of this result relies on a qualitative estimate of continuation from boundary data for the wave equation. The result for the wave equation yields a log-log type stability estimate when the traces of the normal derivatives of the normalized eigenfunctions are known only on a part of the boundary. Recently, an extension of this result to a log type stability estimate was proved in \cite{BCY2} under an additional assumption in terms of the X-ray transform of the potential. The general problem of determining the potential from Dirichlet eigenvalues and the traces of the normal derivatives of the normalized eigenfunctions on a part the boundary is still an open problem. 
An earlier work by Isakov and Sun \cite{IsakovSun92} proves stability estimates for a partial hyperbolic Dirichlet-to-Neumann map  in the special case where the measurements are made in the intersection of the boundary of the domain with a half space. They establish a H\"older type stability estimate in dimension three and a logarithmic type stability estimate in dimension two.

An inverse spectral problem with  different spectral data was considered by Kurylev, Lassas and Weder \cite{KurylevLW_2005}. The case of a singular potential was considered by P{\"a}iv{\"a}rinta and Serov \cite{MR1943362}, following the approach  by Nachman, Sylvester and Uhlmann \cite{NachmanSU_Borg-Lev}. 

 In \cite{Isozaki89, Isozaki91}, Isozaki proved that if we are missing a finite number of eigenvalues and eigenfunctions traces, then we can still determine $q$ uniquely. He also mentioned that in fact, a sharp  enough asymptotic formula would be enough, see also Theorem~\ref{thm_uniqueness} below. 

The purpose of this paper is to prove that even if a finite number of spectral data is missing, we still have conditional H\"older stability, see Theorem~\ref{thm1.1M}. In fact, in Theorem~\ref{thm1.2M}, we prove something more: if the spectral data are known only asymptotically, with a sharp enough estimate of the remainder, similarly to Theorem~\ref{thm_uniqueness}, then the asymptotic data determine the potential in a H\"older stable way.  Theorem~\ref{thm1.2M} combines the previous two theorems but it has stronger assumptions. 

Our work is inspired by the method introduced by Isozaki in \cite{Isozaki89}. It is built on high frequency asymptotics   techniques. The main advantage of this approach is that it is a direct method. In other words, it is not necessary to relate the spectral problem to an hyperbolic Dirichlet-to-Neumann map for the wave equation. Moreover, the results we obtain  this way are stronger than the preceding ones.

\section{Main Results} 
Let $\Omega$ be a bounded domain of $\mathbb{R}^n$ that we assume, for simplicity, of class $C^\infty$. Its boundary will be denoted by $\Gamma$. If $q\in L^\infty (\Omega )$, we denote by $A(q)$ the unbounded operator acting on $L^2(\Omega )$ as follows
$$
A(q)=-\Delta +q,\quad D(A)=H_0^1(\Omega )\cap H^2(\Omega ).
$$
We recall that the spectrum of $A(q)$ consists in a sequence of eigenvalues, counted according to their multiplicity. This sequence can be ordered in the following way :
$$
-\infty <\lambda _1(q)\leq \lambda _2(q)\leq \ldots \leq \lambda _k(q) \leq \ldots  \rightarrow +\infty .
$$
In the sequel we will use the notation $\lambda (q)=\{\lambda _k(q)\}$. We note that as an immediate consequence of the classical min-max principle, we have $\lambda (q)\in \lambda (0)+\ell ^\infty$, where $\ell ^\infty$ is the usual Banach space of bounded sequences equipped with its natural norm. We denote by $ \{\phi_k(q)\}$ an orthonormal set of eigenfunctions, each one related to $\lambda_k(q)$. Note that $\phi_k(q)$ is defined only up to a factor of modulus $1$ when $\lambda_k(q)$ is a simple eigenvalue; and when $\lambda_k(q)$ is multiple, then we have more freedom, and the natural way is to think about eigenspaces. Within each such eigenspace there are infinitely many choices of eigenfunctions and to get the best results in the theorems involving their traces on $\bo$, in \r{thm1.1}, 
we should minimize (take infimum) over all such possible choices.

Our first result is a uniqueness theorem  under the assumption that the spectral data are asymptotically ``very close''. As a partial case, we recover the result in \cite{Isozaki89}  about uniqueness with a finite number of spectral data missing. It was noted in \cite{Isozaki89} that such a result should hold.

\begin{theorem}[Uniqueness] \label{thm_uniqueness}
Let $q_{1,2}\in L^\infty(\Omega)$. 
Let, for some $A>0$, and all $k=1,2\dots$,
\be{thm1.1}
\begin{split}
|\lambda_k(q_1)-\lambda_k(q_2)|&\le Ak^{-\alpha}, \quad \alpha>1,\\
\|\partial_\nu \phi_k(q_1) - \partial_\nu\phi_k(q_2))\|_{L^2(\bo)} &\le Ak^{-\beta},\quad \beta> 1-\frac1{2n}.
\end{split}
\ee
Then $q_1= q_2$. 
\end{theorem}

We shall use the following useful upper and lower bounds for eigenvalues. Let $M>0$ be given. Then there exist two $c_\ast >0$ and $c^\ast >0$, depending only on $M$ and $\Omega$ such that, for all $q\in L^\infty (\Omega )$ satisfying $\|q\|_{L^\infty (\Omega )}\leq M$,
\begin{equation}\label{1.1}
c_\ast k^{2/n}\leq \lambda _k(q)\leq c^\ast k^{2/n},\; k\geq 1.
\end{equation}

From the usual elliptic regularity estimate we have : for any $\epsilon>0$, there exists a constant $C_\epsilon$ 
\begin{equation}\label{1.2}
\|\partial _\nu \varphi _k(q)\|_{L^2(\Gamma )}\leq \|\partial _\nu \varphi _k(q)\|_{H^s(\Gamma )}\leq C_\epsilon\lambda _k(q)^{3/4+\epsilon/2},
\end{equation}
for any $q\in L^\infty (\Omega )$ satisfying $\|q\|_{L^\infty (\Omega )}\leq M$. Here the constant $C$ depends only on $\Omega$ and $M$.

Let 
\be{m}
m>n/2+3/4
\ee
be a fixed integer. Using estimate \eqref{1.1} in \eqref{1.2}, we easily obtain that the sequence $\{k^{-2m/n} \|\partial _\nu \varphi _k(q)\|_{L^2(\Gamma )}\}$ belongs to $\ell ^1$, the Banach space of sequences such that the corresponding series are absolutely convergent. Since $L^2(\Gamma )$ is a Banach space, this is equivalent  to say that the series $\sum k^{-2m/n} \partial _\nu \varphi _k(q)$ converges in $L^2(\Gamma )$.

Let $N\geq 0$ be a fixed integer, we set $\tilde{\lambda}(q)=\{ \lambda _{k+N}(q)\}$ and
$$
\delta _0(q_1,q_2)=\|  \tilde{\lambda}(q_1)-\tilde{\lambda}(q_2)\|_{\ell ^\infty}.
$$
We consider the quantity
$$
\delta _1(q_1,q_2)=  
\sum_{k\geq 1}k^{-2m/n}\| \partial _\nu \varphi _{k+N}(q_1)-\partial _\nu \varphi _{k+N}(q_2)\|_{L^2(\Gamma )}. 
$$

Our second main result is the following stability theorem.

\begin{theorem}[Stability] \label{thm1.1M}
Let $q_1$, $q_2\in L^\infty (\Omega )$ such that $q_1-q_2\in H_0^1(\Omega )$ and
$$
\|q_1\|_{L^\infty (\Omega )}+\|q_2\|_{L^\infty (\Omega )}+\|q\|_{H_0^1(\Omega )}\leq M.
$$
Then there exists $C=C(n,\Omega ,m, M) >0$ and $0<\gamma =\gamma (n) <1$ such that
$$
\|q_1-q_2\|_{L^2(\Omega )}\leq C\delta ^\gamma,
$$
where $\delta =\delta _0(q_1,q_2)+\delta _1(q_1,q_2)$.
\end{theorem}

We will see in the end of the proof of this theorem that we have an explicit value of $\gamma$ as a function of $n$.

Next theorem combines the results of the previous two in one, see also Remark~\ref{remark3m}. It shows that for stable recovery, we only need to know that the data are asymptotically close. In particular, missing a finite number would not affect the stability but it will affect the constants, of course. 

\begin{theorem} 
\label{thm1.2M}
Let $q_1$, $q_2\in L^\infty (\Omega )$ such that $q:= q_1-q_2\in H_0^1(\Omega )$ and
$$
\|q_1\|_{L^\infty (\Omega )}+\|q_2\|_{L^\infty (\Omega )}+\|q\|_{H_0^1(\Omega )}\leq M.
$$
Fix $m$ satisfying \r{m}. 
Let, for some   $\delta>0$, $A>0$,
\be{31xx}
|\lambda_k(q_1)-\lambda_k(q_2)|\le \delta+Ak^{-\alpha}, \quad k^{-2m/n+1}\|\partial_\nu \phi_k(q_1)-\partial_\nu \phi_k(q_2)\|_{L^2(\bo)} \le\delta+ Ak^{-\alpha},
\ee
with  $\alpha>(4m-1)/(2n)$. 
Then there exists $C=C(n,\Omega, m, A,  \alpha   ,M) >0$ and $0<\gamma =\gamma (n, \alpha) <1$ such that
$$
\|q_1-q_2\|_{L^2(\Omega )}\leq C\delta ^\gamma .
$$
\end{theorem}

\begin{remark}\label{remark3m}
Theorem~\ref{thm1.2M} implies somewhat weaker versions of Theorem~\ref{thm_uniqueness} and Theorem~\ref{thm1.1M}: with $\alpha$ and $\beta$ in Theorem~\ref{thm_uniqueness} required to satisfy stronger assumptions; and with a stronger norm of the traces of the normal derivatives of the eigenfunctions in Theorem~\ref{thm1.1M}.
\end{remark}

\textbf{Acknowledgments:} We thank Gunther Uhlmann for his suggestions during the preparation of this paper.

\section{Preliminaries}

First, we consider a family of Dirichlet to Neumann maps parametrized by the spectral parameter. Let $q\in L^\infty (\Omega )$ and $\lambda \in \rho (A(q))$, the resolvent set of $A(q)$. Following well known results on existence and a priori estimate (e.g. \cite{LonsMagenes_I}), for each $f\in H^{1/2}(\Gamma )$, the following boundary value problem
\be{M1x}
\left\{
\begin{array}{ll}
(-\Delta +q-\lambda )u=0\quad &\mbox{in}\; \Omega
\\
u=f &\mbox{on}\; \Gamma 
\end{array}
\right.
\ee 
has a unique solution $u(q,\lambda )(f)\in H^1(\Omega )$ and $f\rightarrow \partial _\nu u(q,\lambda )(f)$ defines a bounded operator from $H^{1/2}(\Gamma )$ into $H^{-1/2}(\Gamma )$. We denote this operator by $\Lambda (q,\lambda )$. It extends to a meromorphic family with poles at the eigenvalues. 

Let $q_1$, $q_2\in L^\infty (\Omega )$, $\lambda \in \rho (A(q_1))\cap \rho (A(q_2))$ and $f\in H^{1/2}(\Gamma )$. Then $u=u(q_1,\lambda )(f)-u(q_2,\lambda )(f)$ is the solution of the following boundary value problem
\be{M1a}
\left\{
\begin{array}{ll}
(-\Delta +q_1-\lambda )u=(q_2-q_1)u(q_2,\lambda )(f)\quad &\mbox{in}\; \Omega
\\
u=0 &\mbox{on}\; \Gamma 
\end{array}
\right.
\ee
Therefore, according to the classical $H^2$ a priori estimate, $\Lambda (q_1,\lambda )-\Lambda (q_2,\lambda )$ defines a bounded operator from $H^{1/2}(\Gamma )$ into $H^{1/2}(\Gamma )$.

Next,  $\Lambda(q,\lambda)$, $q\in L^\infty (\Omega )$, is symmetric as form on $C^\infty(\bo)\times C^\infty(\bo)$ with respect to the duality form $\langle f,g \rangle= \int_{\bo} f g\,d S_x$. Indeed,
\[
\langle \Lambda(q,\lambda)f,g\rangle = \int_\Omega\left( G\Delta F  - F\Delta G    \right)\,d x = \int_\Omega((q-\lambda) -(q-\lambda)) FG\,d x= 0, 
\]
where  $F=u(q,\lambda)(f) $ and $G=u(q,\lambda)(g)$. Since $\Lambda(q_1,\lambda )- \Lambda(q_2,\lambda ) $ is a bounded operator on $H^{1/2}(\bo)$, its transpose, which is the same operator, is  bounded on $H^{-1/2}(\bo)$. 
By interpolation,
\be{M1}
\Lambda(q_1,\lambda )- \Lambda(q_2,\lambda ): H^{s}(\bo)\to H^{s}(\bo) , \quad |s|\le 1/2
\ee
is bounded, as well. Note that for $q$ smooth, $\Lambda(q,\lambda)$ is a pseudo-differential operator of order $1$, see \cite{LeeU}, while one can see that the difference \r{M1} is of order $-1$ either by \r{M1a} or by calculating the first few terms of the symbol in the spirit of \cite{LeeU}. Then, in particular, \r{M1} can be improved.

The following formal representation of $\Lambda(q,\lambda)$ providing  a relationship between the spectral data and the family of DN maps $\Lambda (q,\lambda )$,  appears in \cite{NachmanSU_Borg-Lev}
\be{8}
\Lambda(q,\lambda)f \quad \text{\rm ``=''} \quad  \sum_{j=1}^\infty\frac1{\lambda-\lambda_j(q)} \partial_\nu\phi_j (q)(f, \partial_\nu\phi_j (q))_{L^2(\bo)}.
\ee
The series on the right hand side is not absolutely convergent in some special cases, at least, even if considered as a form. A possible way to justify it, suggested in \cite{MR1070844}, is to show that some high order formal derivative converges. Set
\[
\Lambda ^{(m)}(q,\lambda ) := \frac{d^m}{d\lambda^m}\Lambda ^{(m)}(q,\lambda ).
\]
It was then shown in  \cite{NachmanSU_Borg-Lev,Ch} that for $m\gg1$, the series converges absolutely because $\lambda_k(q)^{-m-1}\sim k^{-2(m+1)/n}$ decays rapidly when $m\gg1$, while the traces of the eigenfunction on $\bo$ have a fixed polynomial bound. 

\begin{lemma}\label{lemma2.2}
Let $q\in L^\infty (\Omega )$, $\varphi (q)$ an orthonormal basis, $f\in H^{1/2}(\Gamma )$,  $m>n/2+3/4$ and $\lambda \in \rho (A(q))$. Then
\be{l2.2}
\Lambda ^{(m)}(q,\lambda )f=-m!\sum_{k\geq 1}\frac{1}{(\lambda _k(q)-\lambda )^{m+1}}\Big( \int_\Gamma f\partial _\nu \varphi _k(q)\, \d\sigma (x)\Big)\partial _\nu \varphi _k(q),
\ee
where the series converges absolutely in  $H^{1/2}(\Gamma )$ and therefore in $L^2(\Gamma )$.\footnote{ In view of the proof of Lemme~2.28 in \cite{Ch}, we can prove, using a density argument, that the result in Lemma~\ref{lemma2.2} remains valid for $f\in L^2(\Gamma )$.}
\end{lemma}

We can adapt Lemma~2.32 in \cite{Ch} which  goes back to \cite{NachmanSU_Borg-Lev} to the complex case and with $H^{1/2}(\Gamma )$ in place of $H^{3/2}(\Gamma )$. We obtain

\begin{lemma}\label{lemma2.1} 
Under the assumptions of Theorem~\ref{thm1.1M}, for any positive integer $l$ and $0<\epsilon <1/2$, there exists a constant $C_\epsilon >0$, that can depend only on $M$, $\Omega$, $l$ and $\epsilon$, such that
\be{l2.1}
\| \Lambda ^{(j)}(q_1,\lambda )-\Lambda ^{(j)}(q_2,\lambda ) \|_{ \mathcal{L}(H^{1/2}(\Gamma ), L^2(\Gamma )) }\leq \frac{C_\epsilon }{|\Re \lambda |^{j+\sigma _\epsilon}},\; 0\leq j\leq l,\; \Re \lambda \leq -2M,\; \sigma _\epsilon =\frac{1-2\epsilon}{4}.
\ee
\end{lemma}

The importance of this lemma is to provide some decay of the difference $\Lambda(q_1,\lambda )-\Lambda(q_2,\lambda )$, as $\Re\lambda\to -\infty$, in suitable norms; and similarly for the derivatives.  Such an estimate is not surprising, of course, because when $\Re\lambda\to -\infty$, $\lambda$ moves away from the spectrum. 

By Lemma~\ref{lemma2.1},
\be{L}
\Lambda(q_1,\lambda )-\Lambda(q_2,\lambda )= \int_{-\infty}^\lambda \d\lambda_1 \int_{-\infty}^{\lambda_1} \d\lambda_2\dots \int_{-\infty}^{\lambda_{m-1}} \d\lambda_m   
\left(\Lambda ^{(m)}(q_1,\lambda_m )-\Lambda ^{(m)}(q_2,\lambda_m) \right), \quad\lambda\not\in \R,
\ee
in $\mathcal{L}(H^{1/2}(\Gamma ), L^2(\Gamma ))$, where, for non-real $\lambda$, the integrals above are taken over the lines $\Im\lambda_j=\Im\lambda$, $j=1,\dots,m-1$. The estimate \r{l2.1} in Lemma~\ref{lemma2.1}, for $j\ge1$, shows that each integral is absolutely convergent; and the same estimate for $j\ge0$ shows that the initial condition after each integration is zero at $-\infty$. Now, plugging \r{l2.2} into \r{L} provides a direct formula expressing the difference of the DN maps in terms of the spectral data.  On the other hand, integrating term by term is not justified, which is the main reason to work with the differentiated series \r{l2.2}.

At the end of this section, we recall a lemma in \cite{Isozaki89} which is the basis for the proofs of the main theorems.

Let $\varphi_{\lambda ,\omega}(x)=e^{i\sqrt{\lambda}\omega \cdot x}$, $\lambda \in \mathbb{C}\setminus (-\infty ,0]$, with the standard choice of the branch of the square root, and $\omega \in \mathbb{S}^{n-1}$, we consider 
$$
S(q)(\lambda ,\omega ,\theta )=\int_\Gamma \Lambda (q, \lambda )(\varphi_{\lambda ,\omega})\varphi_{\lambda ,-\theta}d\sigma (x),\; \lambda \in \rho(A(q))\setminus (-\infty ,0],\; \omega ,\; \theta \in \mathbb{S}^{n-1}.
$$
 Following Lemma~2.2 in \cite{Isozaki89} we have, for $\lambda \in \rho(A(q))\setminus (-\infty ,0],\; \omega ,\; \theta \in \mathbb{S}^{n-1}$,
 \begin{eqnarray}\label{3.1}
 S(q,\lambda ,\omega ,\theta ) &=& -\frac{\lambda}{2}|\theta -\omega |^2\int_\Omega e^{-i\sqrt{\lambda}(\theta -\omega )\cdot x}dx
 \\
 && +\int_\Omega e^{-i\sqrt{\lambda}(\theta -\omega )\cdot x}q(x)dx-\int_\Omega R(q,\lambda )(q\varphi_{\lambda ,\omega})q\varphi_{\lambda ,-\theta}dx, \nonumber
 \end{eqnarray}
 where $R(q,\lambda )=(A(q)-\lambda )^{-1}$ is the resolvent.
 
 We fix $\xi \in \mathbb{S}^{n-1}$ and $\eta \in \mathbb{S}^{n-1}$, $\eta \bot \xi$. For $\tau >1$, let
 \begin{eqnarray*}
 \left\{
 \begin{array}{ll}
 \theta _\tau = c_\tau \eta +\frac{1}{2\tau}\xi
 \\
 \\
 \omega _\tau = c_\tau \eta -\frac{1}{2\tau}\xi
 \\
 \\
 \sqrt{\lambda _\tau}=\tau +i .
 \end{array}
 \right.
 \end{eqnarray*}
 
 Let $q_1$, $q_2$ satisfy the assumptions of Theorem~\ref{thm1.1M} and let $\varphi (q_1)$, $\varphi (q_2)$ be an arbitrary orthonormal basis.  We fix $0<\epsilon <1/2$ and we set $\sigma =\sigma _\epsilon =(1-2\epsilon )/4$. 
 
 In the sequel $C$ is a generic constant that can depend only on $n$, $\Omega$, $M$ and $\epsilon$. Also, for simplicity, we drop the subscript in $\lambda _\tau$, $\omega _\tau$ and $\theta _\tau$.

Using the classical estimate for the resolvent, where $q=q_1$ or $q_2$,
 $$
 \|R(q,z )\|_{\mathcal{L}(L^2(\Omega ))}\leq \frac{1}{|\Im z|},\; z \not\in \mathbb{R},
 $$
 we easily prove
 \begin{equation}\label{3.2}
\Big|\int_\Omega R(q_1,\lambda )(q_1\varphi_{\lambda ,\omega})q_1\varphi_{\lambda ,-\theta}dx\Big|+\Big|\int_\Omega R(q_2,\lambda )(q_2\varphi_{\lambda ,\omega})q_2\varphi_{\lambda ,-\theta}dx\Big|\leq \frac{C}{\tau}.
\end{equation}  

We deduce from identity \eqref{3.1} and estimate \eqref{3.2}, where the extension by zero outside $\omega$ of $q$ is still denoted by $q$,
\begin{equation}\label{3.3}
|(\hat q_1-\hat q_2)(\xi +\frac{i}{\tau }\xi )|\leq \frac{C}{\tau}+|S(q_1,\lambda , \theta ,\omega )-
S(q_2,\lambda , \theta ,\omega )|.
\end{equation}
Since $\hat q_1-\hat q_2$ is an entire function, uniqueness   would follow if we can show that the difference in the right hand side tends to $0$, as $\tau\to\infty$. For stability, we need to estimate the same difference in terms of the spectral data.

\section{Proof of the uniqueness result} 
In this section, we prove Theorem~\ref{thm_uniqueness}.

\begin{proof}[Proof of Theorem~\ref{thm_uniqueness}] 
Notice first that
\be{11}
\phi_{\lambda,\omega}(x) = e^{(\i\tau-1)\omega\cdot x}\quad \Longrightarrow\quad \left| \phi_{\lambda,\omega}(x)\right|\le C,\quad x\in\Omega. 
\ee
By \r{M1}, the difference $\Lambda(q_1,\lambda )- \Lambda(q_2,\lambda )$  is bounded in $L^2(\bo)$. 

We  get from \r{3.3},
\be{13}
|(\hat q_1-\hat q_2)((1+\i/\tau)\xi)| \le  C \|   \Lambda(q_1,(\tau+i)^2)-  \Lambda(q_2,(\tau+i)^2)  \|_{\mathcal{L}(L^2(\bo))}  +C/\tau.
\ee
where $C$  depends only on an a priori bound for $\|q_1\|_{L^\infty (\Omega)}+\|q_2\|_{L^\infty (\Omega )}$ and $\Omega$. As a corollary, if 
\be{15}
 \|\Lambda(q_1,(\tau+i)^2)-  \Lambda(q_2,(\tau+i)^2)  \|_{\mathcal{L}(L^2(\bo))} = o(1),\quad \text{as $\tau\to\infty$}, 
\ee
then $q_1= q_2$.

Let us see what assumptions would imply \r{15}. Let $f$ and $g$ be fixed with
\[
\|f\|_{L^2(\bo)}=1, \quad \|g\|_{L^2(\bo)}=1.
\]
We have, formally,
\be{16}
\begin{split}
\langle ( \Lambda(q_1,\lambda)- \Lambda(q_2,\lambda))f,g\rangle &= 
\sum_j \frac{1}{\lambda-\lambda_j(q_1)}\langle \partial_\nu \phi_j(q_1),g\rangle \langle  \partial_\nu \phi_j(q_1),f \rangle \\
&\quad - \sum_j \frac{1}{\lambda- \lambda_j(q_2)} \langle \partial_\nu  \phi_j(q_2),g\rangle  \langle\partial_\nu   \phi_j(q_2),   f\rangle = I_1+I_2+I_3,
\end{split}
\ee
where  
\begin{eqnarray*}
&&I_1f=\sum_j \left( \frac{1}{\lambda-\lambda_j(q_1)}  -  \frac{1}{\lambda-\lambda_j(q_2)} \right) 
\langle \partial_\nu \phi_j(q_1),g\rangle \langle  \partial_\nu \phi_j(q_1),f \rangle ,
\\
&&I_2f=\sum_j  \frac{1}{\lambda-\lambda_j(q_2)}  \left(  \langle \partial_\nu \phi_j(q_1),g\rangle     -  \langle \partial_\nu \phi_j(q_2),g\rangle  \right) 
\langle  \partial_\nu \phi_j(q_1),f \rangle ,
\\
&&I_3f=\sum_j \frac{1}{\lambda-\lambda_j(q_2)}    \langle \partial_\nu \phi_j(q_2),g\rangle  
\left( \langle  \partial_\nu \phi_j(q_2),f \rangle -   \langle \partial_\nu \phi_j(q_1),g\rangle  \right) .
\end{eqnarray*}

From now on, $\lambda=(\tau+i)^2$. For $I_1$, we get
\[
\begin{split}
I_1 &= \sum_j\left( \frac{1}{\lambda-\lambda_j(q_1)} - \frac{1}{\lambda- \lambda_j(q_2)}\right) \langle \partial_\nu \phi_j(q_1),g\rangle \langle  \partial_\nu \phi_j(q_1),f \rangle\\
  &= \sum_j \frac{\lambda_j(q_1)-\lambda_j(q_2)}{(\lambda- \lambda_j(q_2))(\lambda- \lambda_j(q_1) )} \langle \partial_\nu \phi_j(q_1),g\rangle \langle  \partial_\nu \phi_j(q_2),f \rangle .
\end{split}
\]
Note that for any $\eps>0$,
\[
|\langle  \partial_\nu \phi_j,f \rangle |\le \|\partial_\nu\phi_j \|_{L^2(\bo)} \|f\|_{L^2(\bo)}\le CAj^{(3/2+\eps)/n} .
\]
We used the following estimates here:
\be{better}
\|\partial_\nu\phi_j \|_{L^2(\bo)} \le \|\partial_\nu\phi_j \|_{H^{\eps /2}(\bo)}\le C\|\phi_j\|_{H^{3/2+ {\eps /2}}} \le C_\eps \lambda_j^{3/4+ {\eps /2}}\le C_\eps' j^{(3/2+\eps)/n}.
\ee
Therefore,
\be{18}
|I_1|\le C \sum_{j=1}^\infty \frac{j^{-\alpha+(3+2\eps)/n}}{ |(\lambda-  \lambda_j(q_1)) ( \lambda- \lambda_j(q_2))|}  .
\ee
For $\lambda$ fixed and not among the eigenvalues, the series converges when $-\alpha+(3+2\eps)/n-4/n<-1$, i.e., when $\alpha>1-1/n$ (one can always find $\eps\ll1$ so that  $-\alpha+(3+2\eps)/n-4/n<-1$). 
For $\lambda=(\tau+i)^2$, we apply Lemma~\ref{lemma1} below, see also Remark~\ref{remark4} with $\mu= -\alpha+(3+2\eps)/n$, $0<\eps\ll1$ and  $\nu=2$  to conclude that 
\[
|I_1|\le C\tau^{  (1-\alpha)n+2\eps }
\]
when $\alpha\le 1+1/n +2\eps/n$. Choose $\alpha=1+3\eps/n$, $\eps< 1/n$ to get $O(\tau^{-\eps})$ above. 
 
To estimate $I_2$, write
\be{19b}
|I_2|\le C \sum_j \frac{j^{-\beta}j^{(3/2+\eps)/n}  }{|\lambda-  \lambda_j(q_2)|}. 
\ee
We apply Lemma~\ref{lemma1} again,  with $\nu=1$  and $\mu= -\beta+(3/2+\eps)/n$, see also Remark~\ref{remark3}. Since $\beta>1-1/(2n)$, we can always find $0<\epsilon\ll1$ so that $\beta>1-1/(2n)+ {4\eps/n}$.  With that choice of $\beta$, we apply the second or the third inequality in \r{34a}. Note that $\mu<2/n-1$; in fact, $\mu<2/n-1-3\eps/n$, and $(\mu+1)n/2<1-3\eps/2$, see \r{34a}. If, in addition, $-1\le \mu$, we apply the second inequality to get
\[
|I_2|\le C_\eps\tau^{   \eps+ (\mu+1)n/2-\nu}\le C\tau^{-\eps/2}.
\]
If $\mu<-1$, we apply the third inequality in \r{34a} that gives us the even better estimate $I_2=O(\tau^{-1})$.  

 We treat $I_3$ in a similar way.

Therefore, \r{15} is satisfied, so $\hat q_1(\xi)=\hat q_2(\xi)$ for any $\xi$. Therefore, $q_1=q_2$. 
\end{proof}

 Remark. Let us discuss briefly why formula $(4.4)$ is valid under assumption $(2.1)$. To this end we set $I(\lambda )=I_1+I_2+I_3$, where $I_1$, $I_2$ and $I_3$ are as above. Then one can prove in a straightforward way that $I$ is analytic in $\rho (A(q_1))\cap \rho (A(q_2))$. By using Lemma 3.1 and the fact that weak analyticity is equivalent to strong analyticity, we deduce, where $m>n/2+1$ is fixed,
$$
I^{(m)}(\lambda )=\langle (\Lambda ^{(m)}(q_1,\lambda )-\Lambda ^{(m)}(q_2,\lambda ))f,g\rangle ,\;\lambda \in \rho (A(q_1))\cap \rho (A(q_2)).
$$
Therefore
$$
I(\lambda )=\langle (\Lambda (q_1,\lambda )-\Lambda (q_2,\lambda ))f,g\rangle +\sum_{k=1}^{m-1}a_k\lambda ^k,\;\lambda \in \rho (A(q_1))\cap \rho (A(q_2)).
$$ 
From Lemma 3.2 we know that
$$
\lim_{\tau \rightarrow +\infty}\langle (\Lambda (q_1,(\tau +i)^2 )-\Lambda (q_2, (\tau +i)^2 ))f,g\rangle =0.
$$
On the other hand, from the proof of Theorem 2.1, we easily see that
$$
\lim_{\tau \rightarrow +\infty}I((\tau +i)^2)= 0.
$$
Hence $a_k=0$, $0\leq k\leq m-1$.

The following lemma was used in the proof. Its proof, on the other hand is based on Lemma~\ref{lemma2} below. 

\begin{lemma} \label{lemma1}
Let 
\be{34}
I(\lambda)  = \sum_{j=1}^\infty \frac{j^\mu}{\left|  \lambda-\lambda_j(q)\right|^\nu }, \quad \mu\in\R,\; \nu\ge 0.
\ee
Then for $\mu<2\nu/n-1$ and $\lambda=\tau+i$, $\tau>0$,  the series converges absolutely and for any $\eps>0$, $\tau>1$,
\be{34a}
|I((\tau+i)^2)| \le \begin{cases}
 C\tau^{(\mu+1)n-1-\nu}& \text{for $2/n-1  \le \mu$},\\
 C_\eps\tau^{ \eps+ (\mu+1)n/2-\nu}  &\text{for $-1\le \mu < 2/n-1   $},\\
  C\tau^{  -\nu} & \text{for $ \mu < -1   $ }.
\end{cases}
\ee
\end{lemma}

\begin{remark}\label{remark3a}
The condition $2/n-1  \le \mu$ is not compatible with the convergence condition  $\mu<2\nu/n-1$ when $0\le \nu\le1$. In that case, only the second and the third inequalities above are applicable. The case $\nu=1$ was of particular importance above, see \r{19b} and  also Remark~\ref{remark3}.
\end{remark}

\begin{proof} 
We will sketch first the idea. Since $\lambda_j(q)\sim j^{2/n}$, we get
\be{20}
\begin{split}
|I\left((\tau+i)^2\right)|  &\sim  \sum \frac{j^{\mu}}{ \left|\lambda -j^{2/n}\right|^\nu     } = \sum \frac{j^{\mu}}{ \left(|\tau^2-1-j^{2/n}|^2 +4\tau^2 \right)^{\nu/2}   }\\ 
 &\sim I^\sharp := \int_1^\infty \frac{x^\mu}{\left(\big|\tau^2- x^{2/n}\right|^2+4\tau^2 \big)^{\nu/2} }\,d x.
\end{split}
\ee
The series and the integral are convergent, if $\mu - 2\nu/n<-1$. 
Make the change $x^{2/n}=t$ to get
\be{22}
I^\sharp = C\int_1^\infty \frac{t^{b}\, d t}{ \left((t-\tau^2)^2+4\tau^2 \right)^{\nu/2} },  \quad b:= (\mu+1)n/2-1.
\ee
The convergence condition takes the form $b-\nu<-1$. By Lemma~\ref{lemma2} below, for any  $\eps>0$, and for $\mu<-1+2\nu/n$, 
\begin{alignat}{3}\label{l'2a}
          I^\sharp &\le C\tau^{(\mu+1)n-1-\nu}  &\quad &\text{for $0\le (\mu+1)n/2-1$}, &\quad&\text{i.e., for $2/n-1  \le \mu   $},\\
          I^\sharp &\le C_\eps\tau^{ \eps+ (\mu+1)n/2-\nu}  &\quad&\text{for $-1\le  (\mu+1)n/2-1<0$} \label{l'2b},     &\quad&\text{i.e., for $-1\le \mu < 2/n-1   $},\\
          I^\sharp &\le C_\eps\tau^{  -\nu} &\quad&\text{for $ (\mu+1)n/2-1<-1$} \label{l'2c}, &\quad&\text{i.e., for $ \mu < -1   $ }.
\end{alignat}

We proceed to the actual proof now. The only step that needs to by justified is \r{20} above. 
Let $N(r)$ be the counting function of the square roots $\lambda_j^{1/2}(q)$ of the eigenvalues $\lambda_j (q)$ counted with their multiplicities, i.e., $N(r) = \#\{\lambda_j^{1/2} (q)\le r\}$. 
We will use the Weyl asymptotic formula with a sharp estimate of the reminder term  (reference?)
\be{W1}
N(r) = c_n r^n +O(r^{n-1}).
\ee
Another way to write this asymptotic formula is 
\be{W2}
\lambda_j^{1/2}(q)  = \tilde c_n j^{1/n}+O(1).
\ee
As a consequence, 
\be{W2a}
|\lambda_j^{1/2}(q)-\tilde c_n j^{1/n}|\le A_n
\ee
with some $A_n>0$. Moreover, $A_n$ can be chosen independent of $q$ as long as $q$ belongs to a fixed ball in $L^\infty (\Omega )$. 

We start with
\[
|I((\tau+i)^2)| \le \sum_{j=1}^\infty \frac{j^\mu}{\left((\tau^2-1-\lambda_j (q))^2+4\tau^2\right )^{\nu/2} }\le \sum_{j=1}^\infty \frac{j^\mu}{\left((\tau^2-1-\lambda_j(q))^2+4(\tau^2-1)\right )^{\nu/2} }.
\]
Set $\tau_1= \sqrt{\tau^2-1}$, $\tau>1$. Write for simplicity $I := I((\sqrt{\tau_1^2+1}+i)^2)$  to get, after replacing $\tau_1$ by $\tau$ again,
\[
|I|\le \sum_{j=1}^\infty \frac{j^\mu}{\left((\tau^2-\lambda_j(q))^2+4\tau^2)\right )^{\nu/2} }.
\]
Clearly, it is enough to prove the estimates \r{l'2a}--\r{l'2c} for $I$. Split the sum above into three parts; the first, $I_1$, over the terms with $\lambda_j^{1/2}(q)\le \tau-2A_n$; and the second, $I_2$, for $\tau-2A_n\le \lambda_j^{1/2}(q)\le \tau+2A_n$; and the third containing all terms with $\lambda_j^{1/2}(q)>\tau+2A_n$. When $\lambda_j^{1/2}(q)\le \tau-2A_n$, we have 
\be{22a1}
|\tau^2-\lambda_j(q)| = \tau^2-\lambda_j(q) \ge \tau^2 -(\tilde c_nj^{1/n}+A_n)^2.
\ee
On the other hand, $\tilde c_nj^{1/n}+A_n\le \lambda_j^{1/2}(q) +2A_n\le\tau$, therefore, the term on the right above is non-negative. Then
\be{22a}
(\tau^2-\lambda_j (q))^2 \ge \left(\tau^2 -(\tilde c_nj^{1/n}+A_n)^2\right)^2.
\ee
When  $\lambda_j^{1/2}(q) >\tau+ 2A_n$, we have 
\be{22b1}
|\tau^2-\lambda_j(q)| = \lambda_j(q)-\tau^2 \ge   (\tilde c_nj^{1/n}-A_n)^2 - \tau^2.
\ee
Similarly, $\tilde c_nj^{1/n}-A_n\ge \lambda_j^{1/2}(q) - 2A_n\ge\tau$, therefore, the term on the right above is non-negative and we can take squares of both sides of the inequality to get
\be{22b}
(\tau^2-\lambda_j(q))^2 \ge \left(\tau^2 -(\tilde c_nj^{1/n}-A_n)^2\right)^2.
\ee
Therefore,
\be{23a}
\begin{split}
I_1 &\le \sum_{\lambda_j^{1/2}(q)\le \tau-2A_n}\frac{j^\mu}{\left((\tau^2- (\tilde c_n j^{1/n}+A_n)^2)^2+4\tau^2\right )^{\nu/2} }\\
I_3 &\le \sum_{\lambda_j^{1/2}(q)\ge \tau+2A_n} \frac{j^\mu}{\left((\tau^2- (\tilde c_n j^{1/n}-A_n)^2)^2+4\tau^2\right )^{\nu/2} }.
\end{split}
\ee
Since the r.h.s.\ of \r{22a1} is non-negative, each summand  the first series in \r{23a} is an increasing function of $j$. Therefore, it can be estimated by above by the integral
\[
I_1^\sharp := \int_1^\infty\frac{ y^\mu  \,d y}{  ((\tau^2-(\tilde c_ny^{1/n}+A_n)^2)^2+4\tau^2)^{\nu/2 }}.
\]
The change $\tilde c_n y^{1/n}+A_n=x^{1/n}$ reduces the analysis of $I_1^\sharp$ to that of  \r{22}. Then we use  Lemma~\ref{lemma2} to estimate $I_1^\sharp$, to conclude that the estimates \r{34a} in  Lemma~\ref{lemma1} apply to $I_1^\sharp$, and therefore to $I_1$, as well.  

We analyze $I_3$ in a similar way to conclude that the estimates \r{34a} apply to  $I_3$, as well. 

The analysis of $I_2$ is straightforward. In the interval $[\tau-2A_n,\tau+2A_n]$ there are $O(\tau^{n-1})$ square roots of eigenvalues $\lambda_j$'s. 
The numerator admits the estimate $j^\mu\le C\lambda_j^{n\mu/2}(q)\le C'\tau^{n\mu}$. The denominator can be estimated by below by $C\tau^\nu$. Therefore,
\[
I_2\le C\tau^{n-1+n\mu - \nu},
\]
which is as in \r{l'2a}, and stronger than \r{l'2b}, \r{l'2c}. This completes the proof of the lemma.
\end{proof} 

\begin{remark}\label{remark4}
The only property of the eigenvalues that we used was the estimate \r{W2a}. The lemma remains true if we replace $(\lambda-\lambda_j(q))^\nu$ by the product $(\lambda-\lambda_j(q_1))^{\nu_1}(\lambda-\lambda_j(q_2))^{\tilde \nu_2}$, $\nu_1+\nu_2=\nu$, where $\nu_{1,2}\ge0$, $\nu=\nu_1+\nu_2$. 
\end{remark}

We used the following lemma in the proof. 
\begin{lemma}\label{lemma2} Let $I^\sharp$ be as in \r{22}. Then 
Let $b-\nu<-1$. Then, for all $\eps>0$,
\begin{alignat}{2}\label{l2a}
I^\sharp &\le C\tau^{2b+1-\nu},  &\quad &\text{for $0\le b$},\\
I^\sharp &\le C_\eps\tau^{ \eps+b+1-\nu},  &\quad&\text{for $-1\le  b<0$} \label{l2b},\\
I^\sharp &\le C_\eps\tau^{  -\nu},  &\quad&\text{for $ b<-1$} \label{l2c}.
\end{alignat}
\end{lemma}

\begin{remark}\label{remark3}
Computer algebra calculations in a few special cases show that \r{l2a} is sharp at least for $\nu=2$, while \r{l2b}, \r{l2c} are not when $\nu=1$; the estimate in case \r{l2b} then seems to hold as in \r{l2a}, with a possible logarithmic term. The only case in the proof of the theorem above when $b$ is not strictly positive is when $\nu=1$ and $b<0$, small; and we want to have $I^\sharp = o(1)$. We cannot apply \r{l2a} then, and we use \r{l2b}. The loss of the $\tau^b$ factor in \r{l2a} compared to \r{l2b} for $0<-b\ll1$ then is not essential. 
\end{remark}

\begin{proof}
Perform another  change of variables $t=\tau s+\tau^2$. Assume in what follows that $\tau>1$.  We get
\be{24}
I^\sharp = \int_{\tau^{-1}-\tau}^\infty \frac{(\tau s+\tau^2)^{b     } \tau \,d s}{ (\tau^2s^2+4\tau^2)^{\nu/2}  } = 
\tau^{2b+1-\nu  } \int_{\tau^{-1}-\tau}^\infty \frac{( s/\tau+1 )^{ b}}{ (s^2+4)^{\nu/2} }\, d s.
\ee
For $s\ge\tau^{-1}-\tau$,   and $b\ge0$, we have
\be{24a}
(s/\tau+1)^b \le \left(|s|/\tau+1\right)^b\le  (|s|+1)^b.
\ee
When $b<0$, we have $s/\tau+1>\tau^{-2}$, therefore,
\be{24b}
(s/\tau+1)^b  \le \tau^{2|b|}.
\ee
Therefore, by \r{24a},
\be{25}
I^\sharp \le \tau^{2b+1-\nu } \int_\R\frac {  (|s|+1)^{ b }}{(s^2+1)^{\nu/2}}\,d s\le C \tau^{2b+1-\nu}  
\quad \text{for $0\le b< \nu-1$}.
\ee
On the other hand, for $b<0$ we get by \r{24b},
\be{26}
I^\sharp \le C\tau^{1-\nu} \quad \mbox{for}\;  b<0,\, 1<\nu.
\ee

This is weaker than what we need to prove \r{l2b}. 
To get the stronger estimates in the lemma, we apply the H\"older inequality to \r{22} as follows. 
Write $t^b= t^{b_1}t^{b_2}$ with $b_{1,2}<0$, $b=b_1+b_2$. Then, for $1/p_1+1/p_2=1$, $p_1>1$, $p_2>1$, we have
\be{30}
I^\sharp \le  \left(\int_1^\infty \frac{t^{-p_1|b_1|}\, d t}{ \left((t-\tau^2)^2+4\tau^2 \right)^{p_1\nu/2} }  \right)^{1/p_1} 
\left(\int_1^\infty  {t^{-p_2|b_2|}\, d t}  \right)^{1/p_2}.
\ee
The first integral is the same as $I^\sharp$ but with $\tilde b=p_1b_1$, $\tilde \nu=p_1\nu$. Since $b-\nu<-1$, and $p_1>1$, we have $p_1b-p_1\nu<-1$, therefore that integral converges again. By \r{26},  if, $\tilde \nu >1$,  it is $O(\tau^{1-\tilde\nu})$. 
If $p_2|b_2|>1$, then the second integral in \r{30} is convergent. Therefore,
\be{31}
I^\sharp \le C\tau^{(1-p_1 \nu)/p_1} = C\tau^{1/p_1-\nu}.
\ee
It remains to see what choice of $p_1$, $p_2$, $b_1$, $b_2$ satisfying the conditions above gives the best decay. We need to minimize (up to an $\eps>0$ error) the factor $1/p_1$. Choose $b_1=\eps b$, $b_2=(1-\eps)b$, with $0<\eps\ll1$. We need to satisfy the conditions 

\[
p_2|b_2|>1, \quad \tilde \nu=p_1\nu>1, \quad p_{1,2}>1,\quad \frac1{p_1}+\frac1{p_2}=1. 
\]

Set $1/p_2 = (1-2\eps)|b|$. If $\eps\ll 1/2$, and $|b|\le1$, we have $p_2>1$ as required. Then $1/p_1 = 1-(1-2\eps)|b|>0$. We also have $p_1>1$. 
Then $1/p_1-\nu= 1-(1-2\eps)|b|-\nu$ can take any value greater than $1-|b|-\nu= 1+b-\nu$. The latter is negative by assumption, so we can take $1/p_1-\nu<0$ when $\eps\ll1$; then $\tilde\nu>1$.
This proves \r{l2b}. The proof of \r{l2c} is immediate since (in view of the definition of $I^\sharp$) the denominator is bounded by below by $(4\tau^2)^{\nu/2}$, while the numerator is integrable over $(1,\infty)$. 
\end{proof}

In next section, we would need an estimate of $I(z)$, as $\Re z<0$.

\begin{lemma} \label{lemma3} 
Let $I(\lambda)$ be as in \r{34}. 
Then for $\mu<2\nu/n-1$, $\nu>0$  and $\Im\lambda<0$  the series converges absolutely and for   $\Im \lambda<-1$,
\be{34a.3}
|I(\lambda)| \le  C|\Im\lambda|^{  (\mu+1)n/2-1-\nu}\le C .
\ee
\end{lemma}

\begin{proof}
For $\Im\lambda<0$, we have $|\lambda-\lambda_j|\ge |\Im\lambda|+\lambda_j$. Therefore, for    $\Im \lambda<-1$,
\[
|I(\lambda)|\le \sum \frac{j^\mu}{ \left( |\Im\lambda|+\lambda_j\right)^\nu } \sim \sum \frac{j^\mu}{ \left( |\Im\lambda|+c_*j^{2/n}\right)^\nu } \sim \int_0^\infty \frac{x^\mu \,d x}{ \left( |\Im\lambda|+c_*x^{2/n}\right)^\nu }. 
\]
Make the change of variables $|\Im\lambda| t = x^{2/n}$ to get easily 
\[
|I(\lambda)|\le C|\Im\lambda|^{(\mu+1)n/2-1-\nu}, \quad \text{if $  \mu-2\nu/n<-1 $}.
\]
\end{proof} 

Remark~\ref{remark4} applies to this case as well.
 
\section{Proof of the stability estimate}
We start with \r{3.3}. Set $q= q_1-q_2$. We have 
$$
|\hat{q}(\xi  )|\leq |\hat{q}(\xi +\frac{i}{\tau }\xi )|+\frac{1}{\tau }\sup_{0\leq s\leq 1}|\nabla \hat{q}(\xi +\frac{is}{\tau }\xi  )\cdot  \xi|.
$$
On the other hand, 
$$
|\partial _j\hat{q}(\xi +\frac{is}{\tau }\xi  )|=|\widehat{x_jq}(\xi +\frac{is}{\tau }\xi  )|\leq ce^{\frac{c|\xi|}{\tau }}\|q\|_{L^\infty (\Omega )}
$$
with some $c=c(\Omega )$. Hence,
\begin{equation}  \label{3.4}
|\hat{q}(\xi  )|\leq |\hat{q}(\xi +\frac{i}{\tau }\xi )|+\frac{c|\xi |}{\tau }e^{\frac{c|\xi|}{\tau }}\|q\|_{L^\infty (\Omega )}.
\end{equation}
Until the rest of the proof, we choose $\lambda=\tau+i$, $\tau\geq 1$. 
Combining \eqref{3.3} and \eqref{3.4} yields
$$
C|\hat{q}(\xi )|\leq \frac{1}{\tau}+\frac{|\xi |}{\tau }e^{\frac{c|\xi|}{\tau }}+|S(q_1,\lambda , \theta ,\omega )- S(q_2,\lambda , \theta ,\omega )|.
$$
Integrate this in the ball $|\xi|\le r^\alpha$ to get
\begin{eqnarray*}
C\int_{|\xi |\leq \tau ^\alpha}|\hat{q}(\xi )|^2d\xi &\leq & \frac{\tau ^{\alpha n}}{\tau ^2}+\frac{\tau ^{\alpha (2+n)}}{\tau ^2}e^{\frac{c\tau ^\alpha}{\tau }}\\ &&\qquad +\tau ^{\alpha n}|S(q_1,\lambda , \theta ,\omega )-
S(q_2,\lambda , \theta ,\omega )|^2.
\end{eqnarray*}
Choose  $\alpha =1/(2+n)$ to get
\begin{equation} \label{3.5}
C\int_{|\xi |\leq \tau ^{1/(2+n)}}|\hat{q}(\xi )|^2d\xi \leq \frac{1}{\tau }+\tau ^{n/(n+2)}|S(q_1,\lambda , \theta ,\omega )-
S(q_2,\lambda , \theta ,\omega )|^2.
\end{equation}
Since $q\in H^1(\mathbb{R}^n)$, we have
\begin{eqnarray*}
\|q\|_{L^2(\Omega )}^2&=& \|q\|_{L^2(\Omega )}^2
\\
&=& \int_{|\xi |\leq \tau ^{1/(n+2)}}|\hat{q}(\xi )|^2d\xi +\int_{|\xi |> \tau ^{1/(n+2)}}|\hat{q}(\xi )|^2d\xi 
\\
&\leq & \int_{|\xi |\leq \tau ^{1/(n+2)}}|\hat{q}(\xi )|^2d\xi +\frac{1}{\tau ^{2/(n+2)}}\int_{|\xi |> \tau ^{1/(n+2)}}|\xi |^{2}|\hat{q}(\xi )|^2d\xi 
\\
&\leq & \int_{|\xi |\leq \tau ^{1/(n+2)}}|\hat{q}(\xi )|^2d\xi +\frac{1}{\tau ^{2/(n+2)}}\|q\|_{H^1(\mathbb{R}^n )}^2.
\end{eqnarray*}
Then it follows from \eqref{3.5}
\begin{equation} \label{3.6}
C\|q\|_{L^2(\Omega )}^2\leq \frac{1}{\tau ^{2/(n+2)}}+\tau ^{n/(n+2)}|S(q_1,\lambda , \theta ,\omega )-
S(q_2,\lambda , \theta ,\omega )|^2.
\end{equation}
We now estimate the last term in the above inequality in terms of $\|\Lambda  (q_1,\lambda )-\Lambda  (q_2,\lambda )\|$. We have
\begin{eqnarray*}
&& |S(q_1,\lambda, \theta ,\omega )- S(q_2,\lambda, \theta ,\omega )|= \Big|\int_\Gamma \varphi_{\lambda,-\theta} [\Lambda  (q_1, \lambda )-\Lambda  (q_2, \lambda )]\varphi_{\lambda,\omega} \, d\sigma\Big|
\\
&& \qquad \leq C(\Omega )\|\Lambda  (q_1, \lambda )-\Lambda  (q_2, \lambda )\|\, \|\varphi_{\lambda,\omega}\|_{H^{1/2}(\Gamma )}\|\varphi_{\lambda,-\theta}\|_{L^2(\Gamma )}.
\end{eqnarray*}
Here, and in the rest of the proof,  $\|\cdot \|$ denotes the norm in $\mathcal{L}(H^{1/2}(\Gamma ), L^2(\Gamma ))$. 
Since
$$
\|\varphi_{\lambda,\omega}\|_{H^{1/2}(\Gamma )}\leq C\tau ^{1/2},\quad \|\varphi_{\lambda,-\theta}\|_{L^2(\Gamma )}\leq C,
$$
we deduce easily from the last inequality
$$
|S(q_1,\lambda , \theta ,\omega )-
S(q_2,\lambda , \theta ,\omega )|
 \leq C\tau ^{1/2}\|\Lambda  (q_1, \lambda )-\Lambda  (q_2, \lambda )\|.
$$
By \eqref{3.6}, 
\begin{equation} \label{3.7}
C\|q\|_{L^2(\Omega )}^2\leq \frac{1}{\tau ^{2/(n+2)}}+\tau ^{n/(n+2)+1}\|\Lambda  (q_1, \lambda )-\Lambda  (q_2, \lambda )\|^2.
\end{equation}

For $q=q_1$ or $q_2$, we decompose $\Lambda (q,\lambda )$ in the following form $$\Lambda (q, \lambda )=\tilde {\Lambda} (q, \lambda )+\hat{\Lambda}(q,\lambda ),$$
where, for $f\in H^{1/2}(\Gamma )$,
\begin{eqnarray}\label{3.7a}
&&\tilde {\Lambda} (q, \lambda )f=\partial _\nu \Big( \sum_{k>N}\frac{1}{\lambda _k(q)-\lambda}\Big( \int_\Gamma f\partial _\nu \varphi _k(q)d\sigma (x)\Big) \varphi _k(q)\Big) \in H^{-1/2}(\Gamma ) \; (\mbox{formally}),
\\ \label{3.7b}
&&\hat{\Lambda} (q, \lambda )f=\sum_{k\le N}\frac{1}{\lambda _k(q)-\lambda}\Big( \int_\Gamma f\partial _\nu \varphi _k(q)d\sigma (x)\Big) \partial _\nu \varphi _k(q).
\end{eqnarray}
If $c=c(\Omega ,M)$ is a constant such that $\lambda _k(q)\leq ck^{2/n}$, for all $k\geq 1$, then in a straightforward manner we obtain
\begin{equation} \label{3.8}
\| \hat{\Lambda} ^{(j)}(q,z)\|\leq \frac{C}{|\Re z |^{j+1}},\; \Re z \geq 2cN^{2/n}\; \mbox{or}\; \Re z <0.
\end{equation}
In particular, \eqref{3.8} implies
\begin{equation} \label{3.9}
\| \hat{\Lambda} (q, \lambda )\|\leq \frac{C}{\tau ^2},\; \tau \geq \tau _0,
\end{equation}
for some $\tau _0\geq 1$ depending on $N$ and $\Omega$.

In view of \eqref{3.7}, estimate \eqref{3.9} implies 
$$
C\|q\|_{L^2(\Omega )}^2\leq \frac{1}{\tau ^{2/(n+2)}}+\frac{1}{\tau ^{2(n+3)/(n+2)}}+\tau ^{n/(n+2)+1}\|\tilde{\Lambda}(q_1,\lambda )-\tilde{\Lambda}(q_2,\lambda )\|^2.
$$
Hence,
\begin{equation} \label{3.10}
C\|q\|_{L^2(\Omega )}^2\leq \frac{1}{\tau ^{2/(n+2)}}+\tau ^{n/(n+2)+1}\|\tilde{\Lambda}(q_1,\lambda )-\tilde{\Lambda}(q_2,\lambda )\|^2.
\end{equation}

Let ${m=[n/2+3/4]+1}$, where $[n/2+3/4]$ is the entire part of $[n/2+3/4]$.  We will use the estimate from Lemma~\ref{lemma2.1}. More precisely, we have
\begin{equation} \label{3.11}
\| \Lambda ^{(j)}(q_1, z )-\Lambda ^{(j)} (q_2,z)\|\leq \frac{C}{|\Re z |^{j+\sigma}},\; \lambda \in \mathbb{C},\; \Re z \leq -2M,\; 0\leq j\leq m.
\end{equation}
Here $\sigma =(1-2\epsilon )/4$ for some fixed $\epsilon$, $0<\epsilon <1/2$.

It follows from \eqref{3.8} that this estimate holds true if we replace $\Lambda (q_i,\cdot  )$ by $\tilde{\Lambda} (q_i,\cdot )$. That is we have
\begin{equation} \label{3.12}
\| \tilde{\Lambda} ^{(j)}(q_1,z )-\tilde{\Lambda} ^{(j)}(q_2, z )\|\leq \frac{C}{|\Re z|^{j+\sigma}},\; z\in \mathbb{C},\; \Re \lambda \leq -2M,\; 0\leq j\leq m.
\end{equation}

In the sequel, changing $\tau _0$ if necessary, we can assume that $\tau _0^2-1\geq 2M$. For $\rho \geq 2\Re \lambda $, we set $\tilde{\lambda}=- \rho +\lambda$. From Taylor's formula, we have, for $q=q_1$ or $q_2$,
\begin{equation} \label{3.13}
\tilde{\Lambda} (q, \lambda )=\sum_{k=0}^{m-1}\frac{(\lambda -\tilde{\lambda}) ^k}{k!}\tilde{\Lambda} ^{(k)}(q,\tilde{\lambda})+\int_0^1\frac{(1-s)^m(\lambda -\tilde{\lambda}) ^m}{(m-1)!}\tilde{\Lambda} ^{(m)}(q, \tilde{\lambda}+s(\lambda -\tilde{\lambda}))ds.
\end{equation}
We introduce the following temporary notations
\begin{eqnarray*}
&&\mathcal{P}(q,\lambda )=\sum_{k=0}^{m-1}\frac{(\lambda -\tilde{\lambda}) ^k}{k!}\tilde{\Lambda} ^{(k)}(q, \tilde{\lambda})
\\
&&
\mathcal{R}(q, \lambda )=\int_0^1\frac{(1-s)^m(\lambda -\tilde{\lambda}) ^m}{(m-1)!}\tilde{\Lambda} ^{(m)}(q,\tilde{\lambda}+s(\lambda -\tilde{\lambda}))ds.
\end{eqnarray*}
Since $\Re \tilde{\lambda}\leq -2M$, a straightforward application of inequality \eqref{3.12} leads to
\begin{equation} \label{3.14}
\| \mathcal{P}(q_1, \lambda )-\mathcal{P}(q_2,\lambda )\|\leq \frac{C}{|\Re \tilde{\lambda} |^\sigma }\leq \frac{C}{\rho ^\sigma }.
\end{equation}

By Lemma~\ref{lemma2.2}, we know that 
$$
\tilde{\Lambda} ^{(m)}(q,z)f=\sum_{k>N}\frac{1}{(\lambda _k(q)-z)^{m+1}}\Big( \int_\Gamma f\partial _\nu\varphi_k(q)d\sigma (x)\Big) \partial _\nu \varphi_k(q),\; z \not\in \rho( A_q).
$$

In the sequel $\mu =\mu (s)=\tilde{\lambda}+s(\lambda -\tilde{\lambda})=\lambda -(1-s)\rho$ and
$$
N(\lambda )=\min\{k\geq N; \; \lambda_{k+1}(q)\geq 2\Re \lambda \}.
$$

For $\Re\lambda\gg1$, we decompose $\tilde{\Lambda} ^{(m)}(q, \mu )f$ as follows
$$
\tilde{\Lambda} ^{(m)}(q,\mu )f=\Lambda _1^{(m)}(q, \mu )f+\Lambda _2 ^{(m)}(q,\mu )f,
$$
where
\begin{eqnarray*}
&&\tilde{\Lambda} _1^{(m)}(q,\mu )f=\sum_{k=N+1}^{N(\lambda )}\frac{1}{(\lambda _k(q)-\mu )^{m+1}}\Big( \int_\Gamma f\partial _\nu\varphi_k(q)d\sigma (x)\Big)  \partial _\nu \varphi_k(q),
\\
&&\tilde{\Lambda} _2 ^{(m)}(q,\mu )f=\sum_{k>N(\lambda )}\frac{1}{(\lambda _k(q)-\mu )^{m+1}}\Big( \int_\Gamma f\partial _\nu\varphi_k(q)d\sigma (x)\Big)  \partial _\nu \varphi_k(q).
\end{eqnarray*}
We have 
$$
\tilde{\Lambda} _1^{(m)}(q_1\mu )f-\tilde{\Lambda} _1^{(m)}(q_2, \mu )f=I_1+I_2+I_3,
$$
where
\begin{eqnarray*}
&&I_1f=\sum_{k=N+1}^{N(\lambda )}\Big[\frac{1}{(\lambda _k(q_1)-\mu )^{m+1}}-\frac{1}{(\lambda _k(q_2)-\mu )^{m+1}}\Big]\Big( \int_\Gamma f\partial _\nu\varphi_k(q)d\sigma (x)\Big)  \partial _\nu \varphi_k(q_1),
\\
&&I_2f=\sum_{k=N+1}^{N(\lambda )}\frac{1}{(\lambda _k(q_2)-\mu )^{m+1}}\Big( \int_\Gamma f(\partial _\nu\varphi_k(q_1)-\partial _\nu \varphi _k(q_2))d\sigma (x)\Big)  \partial _\nu \varphi_k(q_1),
\\
&&I_3f=\sum_{k=N+1}^{N(\lambda )}\frac{1}{(\lambda _{k,q_2}-\mu )^{m+1}}\Big( \int_\Gamma f\partial _\nu\varphi_k(q_2)d\sigma (x)\Big) \Big[ \partial _\nu \varphi_k(q_1)-\partial _\nu \varphi_k(q_2)\Big].
\end{eqnarray*}

For $\beta >4/n+1$  fixed, we get easily
$$
\|I_1\| \leq  \frac{N(\lambda )^\beta }{|\Im \lambda |^{m+2}}\delta_0(q_1,q_2) \sum_{k=N+1}^{N(\lambda )}k^{-\beta}\| \varphi _k(q_1) \|_{L^2(\Gamma )}^2.
$$
As $\| \varphi _k(q_1) \|_{L^2(\Gamma )}^2\leq Ck^{4/n}$, we have
$$
\sum_{k=N+1}^{N(\lambda )}k^{-\beta}\| \varphi _k(q_1) \|_{L^2(\Gamma )}^2\leq \sum_{k\geq 1}k^{-\beta +4/n}.
$$
Therefore
$$
\|I_1\| \leq C\frac{N(\lambda )^\beta }{|\Im \lambda |^{m+2}}\delta_0(q_1,q_2).
$$
Similarly, we obtain
$$
\|I_2\|+\|I_3\|\leq C\frac{N(\lambda )^{2(m+1)/n}}{|\Im \lambda |^{m+1}}\delta _1(q_1,q_2).
$$
Setting 
$$\delta =\delta_0(q_1,q_2) +\delta_1(q_1,q_2) ,$$
 we deduce
$$
\|I_1\|+\|I_2\|+\|I_3\|\leq C\frac{N(\lambda )^\beta +N(\lambda )^{2(m+1)/n}}{|\Im \lambda |^{m+1}}\delta .
$$ 
Then the choice of $\beta =2(m+5/4)/n$, satisfying $\beta >4/n+1$ by our choice of $m$, yields
 
\begin{equation} \label{3.15}
\|\tilde{\Lambda} _1^{(m)}(q_1,\mu )f-\tilde{\Lambda} _1^{( m)}(q_2,\mu )\|\leq C\frac{N(\lambda )^{{2(m+5/4)/n}}}{|\Im \lambda |^{m+1}}\delta .
\end{equation}
Since
$$
|\lambda _k(q)-\mu|\geq \lambda _k(q)-\Re \lambda +(1-s)\rho \geq \lambda _k(q)-\Re \lambda \geq \frac{\lambda _k(q)}{2},
$$
we can proceed similarly to the case $\Re \lambda \leq -2M$ (see the proof of Lemma 2.33 in \cite{Ch}, p. 72). We find
$$
\|\tilde{\Lambda} _2^{(m)}(q_1,\mu )f-\tilde{\Lambda} _2^{(m)}(q_2,\mu )f\|\leq C\delta .
$$
This and \eqref{3.15} imply
$$
\|\tilde{\Lambda}^{(m)}(q_1,\mu )f-\tilde{\Lambda}^{(m)}(q_2, \mu )f\|\leq CN(\lambda )^{ {2(m+5/4)/n}}\delta .
$$
From the definition of $N(\lambda )$, we have
$$
CN(\lambda )^{2/n}\leq 2 \Re \lambda .
$$
Hence, 
$$
\|\tilde{\Lambda} ^{(m)}(q_1,\mu )f-\tilde{\Lambda}^{(m)}(q_2, \mu )\|\leq C\Re \lambda ^{{m+5/4}}\delta,
$$
and then
$$
\|\tilde{\Lambda} ^{(m)}(q_1, \mu )f-\tilde{\Lambda} ^{(m)}(q_2, \mu )\|\leq C\tau ^{{2(m+5/4)}}\delta .
$$
Therefore
\begin{equation} \label{3.16}
\|\mathcal{R}(q_1,\lambda ) -\mathcal{R}(q_2,\lambda )\|\leq C\rho ^m\tau ^{{2(m+5/4)}}\delta .
\end{equation}
It follows from \eqref{3.13}, \eqref{3.14} and \eqref{3.16} that
$$
C\|\tilde{\Lambda} (q_1, \mu )-\tilde{\Lambda}(q_2, \mu )\|\leq \frac{1}{\rho ^\sigma}+\rho ^m\tau ^{{2(m+5/4)}}\delta.
$$
Plug this estimate in \eqref{3.10} to get 
\begin{equation} \label{3.17}
C\|q\|_{L^2(\Omega )}^2\leq \frac{1}{\tau ^{2/(n+2)}}+\tau ^{{2(n+1)/(n+2)}}\Big(\frac{1}{\rho ^{2\sigma}}+\rho ^{2m}\tau ^{{4(m+5/4)}}\delta ^2\Big).
\end{equation}
Let us choose $\rho =(2\Re \lambda )^\kappa$, with ${\kappa =1/(2\sigma )}$.\footnote{Note that since $\kappa >1$,  $\rho \geq 2\Re \lambda $ is satisfied.}

This choice of $\rho$ in \eqref{3.17} yields 
\begin{equation} \label{3.18}
C\|q\|_{L^2(\Omega )}^2\leq \frac{1}{\tau ^{2/(n+2)}}+\tau ^{{2(n+1)/(n+2)}+4\kappa m+{4(m+5/4)}}\delta ^2.
\end{equation}

A standard minimization argument, with respect to $\tau$, leads
$$
C\|q\|_{L^2(\Omega )}\leq \delta ^\gamma,
$$
with
$$
\gamma =\frac{1}{n+2+2(n+2)(\kappa m+m+5/4)}.
$$
The proof is then complete.

\section{Proof of Theorem~\ref{thm1.2M}}

The starting point is \r{3.7}. Choose $N=N(\delta)$ so that $N^{-\alpha}= \delta$. Then 
\be{m32}
|\lambda_j(q_1)-\lambda_j(q_2)|\le \delta+Aj^{-\alpha}= N^{-\alpha}+Aj^{-\alpha}\le (1+A)j^{-\alpha} \quad \text{for $j\le N$}. 
\ee
On the other hand,
\be{m33}
|\lambda_j(q_1)-\lambda_j(q_2)|\le \delta+Aj^{-\alpha}\le (1+A)\delta \quad \text{for $j\ge N$}. 
\ee
We have analogous inequalities for the expression in \r{31xx} involving the eigenfunctions because the right-hand sides of the two inequalities in \r{31xx} are the same.  
More precisely,
\be{m33phi1}
j^{-2m/n+1}\|\partial_\nu \phi_j(q_1)-\partial_\nu \phi_j(q_2)\|_{L^2(\bo)}   \le (1+A)j^{-\alpha} \quad \text{for $j\le N$},
\ee 
i.e.,
\be{m33phi1a}
\|\partial_\nu \phi_j(q_1)-\partial_\nu \phi_j(q_2)\|_{L^2(\bo)}   \le (1+A)j^{-\alpha  +2m/n-1} \quad \text{for $j\le N$}.
\ee 
Also, 
\be{m33phi2}
j^{-2m/n+1}\|\partial_\nu \phi_j(q_1)-\partial_\nu \phi_j(q_2)\|_{L^2(\bo)} \le (1+A)\delta \quad \text{for $j\ge N$}. 
\ee

Now, define $\hat\Lambda(\lambda)$ and $\tilde\Lambda(\lambda)$ as in \r{3.7a}, \r{3.7b} with that $N$.   While $N=N(\delta)$ depends on $\delta$, the upper bounds in \r{m32} -- \r{m33phi2} do not.

Estimates \r{m32} and \r{m33phi1a} are of the type \r{thm1.1} with the same $\alpha$ and with $\beta = \alpha  -2m/n+1$. By the assumption on $\alpha$, we have $\beta>1-1/(2n)$, as required in \r{thm1.1}. Next, the assumption on $\alpha$ together with \r{m} imply $\alpha>1+1/n$; therefore, the first condition in  \r{thm1.1} holds as well. 
Therefore, by  Theorem~\ref{thm_uniqueness} and its proof, see \r{15}, 
\be{m44}
 \|\hat \Lambda(q_1,(\tau+i)^2)- \hat \Lambda(q_2,(\tau+i)^2)  \| \le C\tau^{-\theta}, \quad \tau>1,
\ee
with some $\theta>0$ depending on $\alpha$. The constant $C$ is independent of $\delta$ because the number of terms in the sum \r{3.7b} does not matter --- we can just complete it to an infinite series by adding  zero terms --- but the upper bound \r{m32} is independent of $\delta$.  
Notice that in \r{15}, the operator norm is in $\mathcal{L}(L^2(\bo), L^2(\bo))$, which is stronger than the $\mathcal{L}(H^{1/2}(\bo), L^2(\bo))$ norm which we use in this proof, denoted simply by $\|\cdot\|$. With $\lambda=\tau+i$, \r{m44} replaces estimate \r{3.9} but we note that in \r{m44}, we have the difference of two $\hat \Lambda$'s, instead of estimating each one. 

Proceeding as in the proof of Theorem~\ref{thm1.1M}, we combine \r{3.7} and \r{m44} to get
\begin{equation} \label{m3.10}
C\|q\|_{L^2(\Omega )}^2\leq \frac{1}{\tau ^{\theta_1}}+\tau ^{n/(n+2)+1}\|\tilde{\Lambda}(q_1,\lambda )-\tilde{\Lambda}(q_2,\lambda )\|^2, \quad \theta_1 := \min(\theta, 2/(n+2)),
\end{equation}
compare to \r{3.10}. 

To get the estimate \r{3.12}, we need an estimate replacing \r{3.8} first. 
We will derive an estimate similar to \r{3.8} for $\Im z<0$ (only) but for the    difference of two $\hat \Lambda$'s, as above. 
Indeed, following the proof of Theorem~\ref{thm_uniqueness}, let us estimate first $I_1$, see \r{18} for $\Im\lambda<0$. By Lemma~\ref{lemma3}, 
$I_1\le   C|\Im\lambda|^{-2+\eps}\le C|\Im\lambda|^{-1}$ by choosing $\eps$ in \r{18} small enough. This proves that the contribution of $I_1$ to \r{3.8} is $O(|\Im\lambda|^{-1})$, when $j=0$, as claimed. When $j\ge1$, we apply the same argument to the differentiated series, similarly to the proof of Theorem~\ref{thm1.1M}. Each differentiation increases $\nu$ in the application of  Lemma~\ref{lemma3} by $1$, and we get $O(|\Im\lambda|^{-1-j})$, as stated. To estimate $I_2$, see \r{19b}, we apply  Lemma~\ref{lemma3} again with $\mu = -\beta+(3/2+\eps)/n$ (and $\beta=\alpha$), and $\nu=1$. The requirement $(\mu+1)n/2<\nu$ is satisfied for $\eps\ll1$, and then we get $I_2\le C|\Im\lambda|^{-1}$. As before, differentiating $\hat \Lambda(\lambda)$, and applying the same argument, we see that each derivative brings another power of $|\Im\lambda|^{-1}$. To summarize, we have
\be{m45}
\|\hat \Lambda^{(j)}(q_1,z) - \hat \Lambda^{(j)}(q_2,z)\|\le \frac{C}{|\Re z|^{j+1}},  \quad \Re z<0.
\ee

Now, using \r{m45} (instead of  \r{3.8}), we prove \r{3.12} in the current setup. 
The rest of the proof is the same  with $\sum_{k=N+1}^{N(\lambda)}$ considered as $\sum_{N<k\le N(\lambda)}$ because $N$ may not be an integer anymore. In \r{3.17}, we get
\begin{equation} \label{m3.18}
C\|q\|_{L^2(\Omega )}^2\leq \frac{1}{\tau ^{\theta_1}}+\tau ^{{2(n+1)/(n+2)}+4\kappa m+{4(m+2)}}\delta ^2.
\end{equation}
The proof follows by a minimization, as before.

%

\end{document}